\documentclass[a4paper]{amsart}

\usepackage{verbatim,amsmath,amssymb}
\usepackage[displaymath,pagewise]{lineno} 

\newtheorem{thm}{Theorem}

\newtheorem{lem}[thm]{Lemma}

\numberwithin{equation}{section} \numberwithin{thm}{section}

\newtheorem{definition}[thm]{Definition}

\newcommand*\patchAmsMathEnvironmentForLineno[1]{%
  \expandafter\let\csname old#1\expandafter\endcsname\csname #1\endcsname
  \expandafter\let\csname oldend#1\expandafter\endcsname\csname end#1\endcsname
  \renewenvironment{#1}%
     {\linenomath\csname old#1\endcsname}%
     {\csname oldend#1\endcsname\endlinenomath}}%
\newcommand*\patchBothAmsMathEnvironmentsForLineno[1]{%
  \patchAmsMathEnvironmentForLineno{#1}%
  \patchAmsMathEnvironmentForLineno{#1*}}%
\AtBeginDocument{%
\patchBothAmsMathEnvironmentsForLineno{equation}%
\patchBothAmsMathEnvironmentsForLineno{align}%
\patchBothAmsMathEnvironmentsForLineno{flalign}%
\patchBothAmsMathEnvironmentsForLineno{alignat}%
\patchBothAmsMathEnvironmentsForLineno{gather}%
\patchBothAmsMathEnvironmentsForLineno{multline}%
}

\newcommand{\dist}{{\rm dist}\,}

\newcommand{\eps}{\varepsilon}

\def\en{\mathbb N}
\def\er{\mathbb R}

\def\H{\mathcal H}
\def\P{\mathcal P}
\def\Q{\mathcal Q}

\newcommand{\co}{\operatorname{co}}

\begin{document}

\title[Nonconvex function which is locally convex outside a discontinuum]{Nonconvex Lipschitz function in plane which is locally convex outside a discontinuum}
\author{Du\v san Pokorn\'y}

\begin{abstract}
We construct a Lipschitz function on $\er^{2}$ which is locally convex on the complement of some totally disconnected compact set but not convex.
Existence of such function disproves a theorem that appeared in a paper by L. Pasqualini and was also cited by other authors.
\end{abstract}
\thanks{The author was supported by a cooperation grant of the Czech and the German science foundation, GA\v CR project no.\ P201/10/J039}
\maketitle

\section{Introduction}

In his work from $1938$ L. Pasqualini presents a theorem (see \cite[Theorem 51, p. 43]{Pa}) of which the following statement is a reformulation:

{\it
Let $f:\er^d\to\er$ be a continuous function and $M\subset\er^d$ a set not containing any continuum of topological dimension $(d-1)$. 
Suppose that $f$ is locally convex on the complement of $M$.
Then $f$ is convex on $\er^d.$
}

The proof however contains a gap. 
This result also appeared in the survey paper \cite{BuZa}, where the (incorrect) proof was shortly repeated.
Also V.G.~Dmitriev mentions this result in \cite{Dm}, although he provides a wrong reference.

As a counterexample to the theorem of Pasqualini we present the following theorem:
\begin{thm}\label{main}
There is a Lipschitz function $f:\er^{2}\to\er$ and $M\subset\er^{2}$ such that
\begin{itemize}
\item $f$ is locally convex on $\er^{2}\setminus M$,
\item $f$ is not convex on $\er^{2}$,
\item $M$ is compact and totally disconnected,
\item $f$ has compact support.
\end{itemize}
\end{thm}

Note that it is simple observation that such set $M$ cannot be of one dimensional Hausdorff measure $0$ (this fact actually essentially follows from the original argument by Pasqualini).

In this situation it seems natural to call a compact set $M$ {\it convex nonremovable} if there is a nonconvex say Lipschitz function $f$ which is locally convex on the complement of $M$.
Note that in such context it may be relevant that the function from Theorem~\ref{main} is Lipschitz (or continuous) or that it has a compact support or that it is defined on whole $\er^2$, since it is possible that such notion of nonremovabiliity might differ if we a priori assume some of those conditions to hold for $f$.
In some sense the set $M$ from Theorem~\ref{main} may be considered as nonremovable in one of the strongest ways possible.

\section{Preliminaries}
In the paper we will use the following more or less standard notation and definitions:

For $a,b\in\er^d$ and $r>0$ we will denote by $B(a,r)$ the closed ball with center $a$ and radius $r$ and $[a,b]$ will denote the closed line segment with endpoints $a$ and $b$. 
For $A\subset\er^d$ the symbol $\co A$ will mean the convex hull of $A$ and $A^c$ will mean the complement of $A.$
If $l\subset \er^2$ is a line and $\eps>0$ then we define $l(\eps)=\{x\in\er^2:\dist(x,l)<\eps\}.$

A function $f$ defined on a set $A\subset\er^2$ is called $L$-Lipschitz, if for every $x,y\in A$, $x\not =y,$ we have $|f(x)-f(y)|\leq L|x-y|.$

We will call $f$ locally convex on $A$ if for every $x,y$ such that $[x,y]\subset A$ and $\alpha\in[0,1]$ we have $f(\alpha x+(1-\alpha)y)\leq \alpha f(x)+(1-\alpha) f(y).$

Finally, $f$ will be called piecewise affine on $A$ if there is a locally finite triangulation $\Delta$ of $A$ such that $f$ is affine on every triangle from $\Delta.$

\section{Construction of the function}

\begin{definition}
Let $\Q$ be a system of all unions of finite systems of (closed) polytopes in $\er^{2}$.
Let $L>0$, $f:\er^{2}\to\er$ and $P\in\Q.$
We say that a pair $(P,f)$ is $L$-good if
\begin{enumerate}
\item $f$ is $L$-Lipschitz,
\item $f$ is piecewise affine on $P^{c}$,
\item $f$ is locally convex on $P^{c}$.
\end{enumerate}
\end{definition}

The key technical result is the following:

\begin{lem}\label{keylemma}
Let $\eps,L>0$, $l$ line in $\er^{2}$ let $(P,g)$ be a $L$-good.
Then there is an $(L+\eps)$-good pair $(Q,h)$ such that
\begin{enumerate}
\item $Q\subset P$,
\item $h=g$ on $P^{c}$,
\item if $x,y\in Q$ belong to the different component of $\er^{2}\setminus l(\eps)$ then they belong to the different component of $Q.$
\end{enumerate}
\end{lem}

We first prove Theorem~\ref{main} using Lemma~\ref{keylemma}

\begin{proof}[Proof of Theorem~\ref{main}]
Choose a sequence $\{x_{n}\}_{n=1}^{\infty}$ dense in the plane and consider any sequence of lines $\{l_{n}\}_{n=1}^{\infty}$ with the property that
for any $i,j\in\en$ there is some $k\in\en$ such that $x_i,x_j\in l_k.$
Choose a sequence $\{\eps_{n}\}_{n=1}^{\infty}\subset(0,\infty)$ such that $\sum_{n=1}^{\infty}\eps_{n}<\infty.$
Then the sequence  $\{l_{n}(\eps_n)\}_{n=1}^{\infty}$ has the property that for every $x,y\in\er^2$, $x\not=y,$ 
there is some $k\in\en$ such that $x$ and $y$ belong to the different component of $\er^2\setminus l_{k}(\eps_k).$

We will proceed by induction and construct a sequence of functions $f_{i}:\er^{2}\to\er$ and a sequence $P_{i}\subset\Q$, $i=0,1,...$, such that for every $i$ the following conditions hold:
\begin{enumerate}
\item pair $(P_{i},f_{i})$ is $(1+\sum_{n=1}^{i}\eps_{n})$-good,
\item if $i>0$ then $P_{i}\subset P_{i-1}$,
\item if $i>0$ then $f_{i}=f_{i-1}$ on $(P_{i-1})^{c}$,
\item if $i>0$ and if $x,y\in P_i$ belong to the different component of $\er^{2}\setminus l_i(\eps_{i})$ then they belong to the different component of $P_{i}.$
\end{enumerate}

To do this let $f_{0}$ be an arbitrary $1$-Lipschitz function on $\er^{2}$ which is equal to $0$ on $((-3,3)^{2})^{c}$ and equal to $1$ on $[-1,1]^{2}$ and put $P_{0}:=[-3,3]^{2}\setminus (-1,1)^{2}.$
Validity of conditions $(1)-(4)$ is obvious.

Now, if we have $f_{i-1}$ and $P_{i-1}$ constructed we obtain $f_{i}$ and $P_{i}$ simply by applying lemma~\ref{keylemma} with
$\eps=\eps_{i},$ $L=(1+\sum_{n=1}^{i-1}\eps_{n}),$ $l=l_{i},$ $P=P_{i-1}$ and $g=f_{i-1}.$
The function $f_{i}$ will be then equal to $h$ from the statement of lemma~\ref{keylemma} and $P_{i}$ will be equal to the corresponding $Q.$
Validity of conditions $(1)-(4)$ follows directly from lemma~\ref{keylemma}.

Put $M:=\cap P_{i}.$
Due to property $(2)$ $M$ is compact and nonempty.
To prove that $M$ is totally disconnected consider $x,y\in M$, $x\not=y.$
By the choice of the sequences $\{l_{n}\}_{n=1}^{\infty}$ and $\{\eps_{n}\}_{n=1}^{\infty}\subset\er^{+}$ there is some $i$ such that
$x$ and $y$ belong to the different component of $\er^{2}\setminus l_{i}(\eps_{i}).$
By property $(3)$ we have that $x$ and $y$ belong to the different component of $P_{i}.$
Using property $(2)$ again we then obtain that $x$ and $y$ belong to the different component of $M$ as well.

Define $\tilde f:M^{c}\to\er$ in such a way that $\tilde f(x)=f_{i}(x)$ whenever $x\in (P_{i})^{c}$.
It is easy to see that the definition of $\tilde f$ is correct due to properties $(2)$ and $(3)$ and the definition of $M$, 
and also that by property $(1)$ the function $\tilde f$ is $(1+\sum_{n=1}^{\infty}\eps_{n})$-Lipschitz and locally convex on $M^{c}$.
By Kirszbraun's theorem there is a $(1+\sum_{n=1}^{\infty}\eps_{n})$-Lipschitz function $f:\er^{2}\to\er$ such that $f=\tilde f$ on $M^{c}.$
Therefore $f$ is locally convex on $M^{c}$ as well. 
Also, $f$ has compact support due to properties $(2)$ and $(3)$, the fact that $P_0$ is compact and that $f_0$ is supported in $P_0.$

It remains to show that $f$ is not convex on $\er^{2},$
but this is easy since
$$
\frac{f(-3,0)+f(3,0)}{2}=0<1=f(0,0).
$$
\end{proof}

The proof of Lemma~\ref{keylemma} is divided into several lemmae.

\begin{lem}\label{ext1}
Let $H\subset\er^2$ be a closed halfplane, $x\in \er^{2}\setminus H$ and $L>0$.
If $f:H\cup \{x\}\to\er$ is $L$-Lipschitz and affine on $H$, then for every $y\in\partial H$ the function
\begin{equation*}
g_y(z) = \left\{
  \begin{array}{l l}
    f(z), &\text{if}\quad z\in H,\\
    \alpha f(x)+(1-\alpha)f(y), & \text{for}\quad z=\alpha x+(1-\alpha)y,\alpha\in[0,1].\\
  \end{array} \right.
\end{equation*}
is $L$-Lipschitz as well.
\end{lem}

\begin{proof}
Without any loss of generality we can suppose that $H=\{(x,y)\in\er^{2}:x\leq 0\}$, $f(y)=0$ and that $y=(0,0).$
This means that $g_y$ is in fact linear on both $H$ and $[x,y].$
Choose $a\in H$ and $b=\alpha x$ for some $\alpha\in[0,1].$
Now,
$$
\begin{aligned}
\left|g_y\left(a\right)-g_y\left(b\right)\right|=&\alpha\left|g_y\left(\frac{1}{\alpha}a\right)-g_y\left(\frac{1}{\alpha}b\right)\right|
=\alpha\left|g_y\left(\frac{1}{\alpha}a\right)-g_y\left(\frac{1}{\alpha}\alpha x\right)\right|\\
=&\alpha\left|g_y\left(\frac{1}{\alpha}a\right)-g_y\left(x\right)\right|
\leq \alpha L\left|\frac{1}{\alpha}a-x\right|
=\alpha L\left|\frac{1}{\alpha}a-\frac{1}{\alpha}\alpha x\right|\\
=&L|a-\alpha x|=L|a-b|.
\end{aligned}
$$
Similarly, if $a=\alpha x$ and $b=\beta x$ for some $\alpha,\beta\in[0,1]$ $\alpha\not=\beta$ we have
$$
|g_y(a)-g_y(b)|=|\alpha f(x)-\beta f(x)|\leq |\alpha-\beta|f(x)\leq|\alpha-\beta|L.
$$
\end{proof}

\begin{lem}\label{square}
Let $\eps,L,K>0.$
Let $f$ be a $L$-Lipschitz function on $[-K,K]^{2}$, which is equal to an affine function $f_{1}$ on $[-K,0]\times[-K,K]$, and $z\in (0,K)\times(-K,K).$
Then there is an $x\in [(0,0),z]$ and $\gamma>0$ such that for every $y\in B(x,\gamma)$ and every $w\in B((0,0),\gamma)\cap (\{0\}\times(-K,K))$ the function 
\begin{equation*}
g_{y,w}(u) = \left\{
  \begin{array}{l l}
    f(u), &\text{if}\quad u\in [-K,0]\times[-K,K],\\
    \alpha f(w)+(1-\alpha)f(y), & \text{for}\quad u=\alpha w+(1-\alpha)y, \alpha\in[0,1]. \\
  \end{array} \right.
\end{equation*}
is $(L+\eps)$-Lipschitz and $|g_{y,w}-f|<\eps$ on $[-K,0]\times[-K,K]\cup [w,y]$.
\end{lem}

\begin{proof}
Without any loss of generality we can suppose that $K=L=1$ and that $f(0,0)=0.$
Since $f$ is $1$-Lipschitz we can find a sequence $\{x_{i}\}_{i=1}^{\infty}\subset[(0,0),z]$ converging to $(0,0)$ such that for some $s\in[-1,1]$
\begin{equation}\label{convergence}
s_{i}:=\frac{f(x_{i})}{|x_{i}|}\to s\quad   \text{as} \quad i\to\infty.
\end{equation}
Consider now the sequence of functions 
$h_{i}:[-\frac{1}{|x_{i}|},0]\times[-\frac{1}{|x_{i}|},\frac{1}{|x_{i}|}]\cup \{\frac{z}{|z|}=:\tilde z\}\to\er$ defined as
$$
h_{i}(u):=\frac{1}{|x_{i}|}f\left(|x_{i}|u\right).
$$
Then $h_{i}$ is $1$-Lipschitz for every $i$.
Since $f$ is equal to an affine function $f_{1}$ on $[-1,0]\times[-1,1]$ and $f(0,0)=0$ we have $h_{i}=f_{1}$ on $[-\frac{1}{|x_{i}|},0]\times[-\frac{1}{|x_{i}|},\frac{1}{|x_{i}|}].$
Also $h_{i}(\tilde z)=s_{i}.$
Therefore by $(\ref{convergence})$ the function $h:(-\infty,0]\times(-\infty,\infty)\cup\{\tilde z\}\to\er$ 
which is equal to $f_{1}$ on $(-\infty,0]\times(-\infty,\infty)$ and such that $h(\tilde z)=s$ is also $1$-Lipschitz.

Consider $\tilde\gamma>0$ such that $\tilde\gamma<\frac{\eps \tilde z_{1}}{4}$ (here by $\tilde z_1$ we mean the first coordinate of $\tilde z$)
and such that $\frac{|v-\tilde z|}{|v-\tilde z|-\tilde\gamma}<1+\frac{\eps}{2}$ for every $v\in (-\infty,0]\times(-\infty,\infty).$

Now, for every $\tilde s\in[s-\tilde\gamma,s+\tilde\gamma]$, $v\in (-\infty,0]\times(-\infty,\infty)$ and $u\in B(\tilde z,\tilde\gamma)$
$$
\begin{aligned}
\frac{f_{1}(v)-\tilde s}{|v-u|}\leq& \frac{|f_{1}(v)-s|}{|v-u|}+\frac{|s-\tilde s|}{|v-u|}
\leq \frac{|f_{1}(v)-s|}{|v-\tilde z|-\tilde\gamma}+\frac{\tilde\gamma}{|v-\tilde z|-\tilde\gamma}\\
\leq& \frac{|f_{1}(v)-s|}{|v-\tilde z|}\cdot\frac{|v-\tilde z|}{|v-\tilde z|-\tilde\gamma}+\frac{2\tilde\gamma}{\tilde z_{1}}
\leq \left(1+\frac{\eps}{2}\right)+\frac{\eps}{2}=1+\eps.
\end{aligned}
$$
Therefore, by lemma~\ref{ext1} for every $\tilde s\in[s-\tilde\gamma,s+\tilde\gamma]$,
$v\in \{0\}\times(-\infty,\infty)$  and $t\in B(\tilde z,\tilde\gamma)$ the function
\begin{equation*}
\tilde h_{v,t,\tilde s}(u) = \left\{
  \begin{array}{l l}
    f_{1}(u), &\text{if}\quad u\in (-\infty,0]\times(-\infty,\infty),\\
    (1-\alpha)\tilde s+\alpha f_{1}(v), & \text{for}\quad u=(1-\alpha)t+\alpha v,\alpha\in[0,1].\\
  \end{array} \right.
\end{equation*}
is $(1+\eps)$-Lipschitz as well.

Choose $i$ such that $s_{i}\in[s-\frac{\tilde\gamma}{2},s+\frac{\tilde\gamma}{2}]$ and put $x=x_{i}$ and $\gamma=\frac{|x|\tilde\gamma}{2}.$
Now, consider some $y\in B(x,\gamma)$ and some $w\in B((0,0),\gamma)\cap \{0\}\times(-1,1)$ and let $g_{y,w}$ be as in the statement on the lemma.
First we will prove that $g_{y,w}$ is $(1+\eps)$-Lipschitz.
To do this we first observe that $\frac{1}{|x|}g_{y,w}(\frac{\cdot}{|x|})$ is equal to $\tilde h_{\frac{w}{|x|},\frac{y}{|x|},\frac{f(y)}{|x|}}(\cdot)$, where the first function is defined.
Now, we have $\frac{w}{|x|}\in \{0\}\times(-\infty,\infty)$,
$$
\left|\frac{y}{|x|}-\tilde z\right|=\left|\frac{y}{|x|}-\frac{x}{|x|}\right|=\frac{|y-x|}{|x|}\leq\frac{|x|\tilde\gamma}{2|x|}\leq\tilde\gamma,
$$
which means $\frac{y}{|x|}\in B(\tilde z,\tilde\gamma)$ and finally
$$
\begin{aligned}
\left|\frac{f(y)}{|x|}-s\right|=&\left|\frac{f(y)-f(x)+f(x)}{|x|}-s\right|\leq \left|\frac{f(y)-f(x)}{|x|}\right|+\left|\frac{f(x)}{|x|}-s\right|\\
\leq& \frac{|y-x|}{|x|}+\frac{\tilde\gamma}{2}\leq\frac{\frac{|x|\tilde\gamma}{2}}{|x|}+\frac{\tilde\gamma}{2}
=\frac{\tilde\gamma}{2}+\frac{\tilde\gamma}{2}=\tilde\gamma.
\end{aligned}
$$
which means that $\frac{f(y)}{|x|}\in[s-\tilde\gamma,s+\tilde\gamma]$ and we are done since $\frac{1}{|x|}g_{y,w}(\frac{\cdot}{|x|})$ and $g_{y,w}$ have the same Lipschitz constant.

To finish the proof it is now sufficient to observe that if we additionally choose $x_i$ small enough we obtain also $|g_{\eps}-f|<\eps$ on $[-1,0]\times[-1,1]\cup [w,y]$.
\end{proof}

\begin{lem}\label{points}
Let $L,\eps,\delta>0$, $a<b$ and $c<d$ be given.
Let 
$$
P=\co\{(-1,a),(-1,b),(1,c),(1,d)\}
$$ 
and 
$$
P^{\eps}=\co\{(-1,a-\eps),(-1,b+\eps),(1,c-\eps),(1,d+\eps)\}.
$$
Suppose that $f$ is a $L$-Lipschitz function defined on $\er^2$ which is locally affine on $P^{\eps}\setminus P.$
Then there are
$$
\frac{a+c}{2}=:a_{0}<a_{1}<...<a_{n-1}<a_{n}:=\frac{b+d}{2}
$$
and $\frac{1}{2}>\kappa>0$
such that, using the notation defined below, the function $g_\kappa:P^{\eps}\setminus (P^{\circ}\setminus[-\kappa,\kappa]\times\er)\to\er$ defined as $g_\kappa(z_i^\pm)=f(z_i^\pm)$ for $i=0,n$, $g_\kappa(z_i^\pm)=f(z_i)$ for $i=1,...,n-1$ and
\begin{equation*}
g_{\kappa}(x) = \left\{
  \begin{array}{l l}
    f(x), &\text{if}\quad x\in P^{\eps}\setminus P,\\
    \alpha g(z_i^+)+\beta g(z_i^-)+\gamma g(z_{i+1}^+), & \text{for}\quad x=\alpha z_{i}^+ +\beta z_{i}^- +\gamma z_{i+1}^+, \\
    &\alpha,\beta,\gamma\geq 0,\alpha+\beta+\gamma=1,\\
    \alpha g(z_i^-)+\beta g(z_{i+1}^-)+\gamma g(z_{i+1}^+), & \text{for}\quad x=\alpha z_{i}^- +\beta z_{i+1}^- +\gamma z_{i+1}^+,\\
    &\alpha,\beta,\gamma\geq 0,\alpha+\beta+\gamma=1
  \end{array} \right.
\end{equation*}
is $(L+\delta)$-Lipschitz and such that $|f-g_{\kappa}|<\delta$ on $\er^2.$

Here we denoted $z_0^\pm:=\left(\pm\kappa, \frac{a+c}{2}\pm\frac{\kappa(a-c)}{2}\right)$, $z_n^\pm:=\left(\pm\kappa, \frac{b+d}{2}\pm\frac{\kappa(b-d)}{2}\right)$,
$z_i^\pm:=(\pm\kappa, a_i)$ for $i=1,...,n-1$ and $z_i:=(0,a_i)$ for $i=0,...,n.$

\end{lem}

\begin{proof}
Without any loss of generality we can suppose $L=1.$
Denote $P^{\eps}_i$ the connectivity component of $P^{\eps}\setminus P^{\circ}$ containing $z_i$, $i=0,n.$
When we will have $a_i$ found we will denote $P_{i}=\co\{c^{\pm}_i,c^{\pm}_{i+1}\}$ for $i=0,...,n-1.$


First use lemma~$\ref{square}$ to find $a_{1}\in B(a_{0},\frac{\min(|a_{0}-a_{n}|,1)}{2})$ and $a_{n-1}\in B(a_{n},\frac{\min(|a_{0}-a_{n}|,1)}{2})$ and $\kappa_1>0$
such that for every $\kappa>0$ the function 
$g|_{P^{\eps}_{0}\cup P_0}$ and $g|_{P^{\eps}_{n}\cup P_{n-1}}$ are both $(1+\delta)$-Lipschitz and such that $|f-g_\kappa|<\delta$ on $P^{\eps}\cup P_0\cup P_{n-1}.$


Observe that for every $u_0\in P^{\eps}_{0}\cup P_0$ and every $u_n\in P^{\eps}_{n}\cup P_{n-1}$ we have
$$
\begin{aligned}
\frac{|g_{\kappa}(u_0)-g_{\kappa}(u_n)|}{|u_0-u_n|}\leq &\frac{|g_{\kappa}(u_0)-g_{\kappa}(z_0)|}{|u_0-u_n|} + \frac{|g_{\kappa}(z_0)-g_{\kappa}(z_n)|}{|u_0-u_n|} + \frac{|g_{\kappa}(z_n)-g_{\kappa}(u_n)|}{|u_0-u_n|}\\
\leq & \frac{|u_0-z_0|}{|u_0-u_n|} + \frac{|z_0-z_n|}{|u_0-u_n|} + \frac{|z_n-u_n|}{|u_0-u_n|}.
\end{aligned}
$$
and since the last formula can be smaller than $1+\delta$ when we assume $|a_0-a_1|$ and $|a_{n-1}-a_n|$ to be small enough, we can additionally assume that
$g|_{P^{\eps}\cup P_0\cup P_{n-1}}$ is $(1+\delta)$-Lipschitz.

Next, note that the function $g_\kappa|_{[z_1,z_{n-1}]}$ is actually independent on $\kappa$ and that it is $1$-Lipschitz for any choice of $a_2,...,a_{n-2}$ 
(this is because in one dimension the affine extension never increases the Lipschitz constant).
This also means that for $S=\co\{c^{\pm}_1,c^{\pm}_{n-1}\}$ we have $g_\kappa|_{S}$ is $1$-Lipschitz for any choice of $a_2,...,a_{n-2}$ as well.
Put $\alpha=\dist(S,P^{\eps}\setminus P)$, we can assume $\kappa_2$ to be small enough that $1>\alpha>0$ (here we used the fact that $|a_0-a_1|,|a_{n-1}-a_n|\leq\frac{1}{2}$).
Consider $n$ big enough such that $\frac{|a_1-a_{n-1}|}{n-1}\leq\frac{\alpha\delta}{4}$, 
put $a_i=a_1+\frac{i|a_1-a_{n-1}|}{n-1}$ and pick $\kappa_3<\min(\kappa_2,\frac{\alpha\delta}{4}).$
Then for $\kappa<\kappa_3$ and $a\in S$
\begin{equation}
\begin{aligned}
|g_{\kappa}(a)-f(a)|&\leq |g_{\kappa}(a)-g_{\kappa}(z_i)|+|g_{\kappa}(z_i)-f(z_i)|+|f(z_i)-f(a)|\\
&\leq |a-z_i|+0+|a-z_i|\leq \frac{\delta}{2}<\delta,
\end{aligned}
\end{equation}
where $i$ is chosen such that $a\in P_i.$

To finish the proof we need to observe that for $\kappa<\kappa_3$ the function $g_\kappa$ is $(1+\delta)$-Lipschitz.
Since $S\cup P_0\cup P_{n-1}$ is convex, the remaining case we have to consider is $a\in S$ and $b\in P^{\eps}\setminus P.$
Find $i$ such that $a\in P_i$.
With this choice we have $|a-z_i|\leq\frac{\alpha\delta}{2}$ and therefore
$$
|b-z_i|\leq|a-b|+|a-z_i|\leq |a-b|+\frac{\alpha\delta}{2}\leq \left(1+\delta\right)|a-b|.
$$
Now,
$$
\begin{aligned}
|g_{\kappa}(a)-g_{\kappa}(b)|\leq &|g_{\kappa}(a)-g_{\kappa}(z_i)|+|g_{\kappa}(z_i)-g_{\kappa}(b)|\\
\leq & \frac{\delta\alpha}{2}+|f(z_i)-f(b)|\leq \frac{\delta}{2}|a-b|+|b-z_i|\\
\leq & \frac{\delta}{2}|a-b|+ \left(1+\frac{\delta}{2}\right) \cdot|a-b|\leq (1+\delta)|a-b|.
\end{aligned}
$$
\end{proof}

\begin{lem}\label{uvnitr}
Let $1>\eps>0$ and $\alpha,L>0.$
Let $f$ be a $L$-Lipschitz function on $[-1,1]^{2}$ which is affine on both $[-1,0]\times[-1,1]$ and $[0,1]\times[-1,1]$ (and equal to affine functions $f_{1}$ and $f_{2}$, respectively).
Put 
$$
A_{1}=[-1,0]\times[-1,-1/2], A_{2}=[0,1]\times[1/2,1], 
$$
$$
B_{1}^{\eps}=[0,\eps]\times[-1,\eps], B_{2}^{\eps}=[-\eps,0]\times[-\eps,1]
$$ 
and 
$$
A=A_{1}\cup A_{2}\cup B_{1}^{\eps}\cup B_{2}^{\eps}.
$$
Then either $f$ is convex on $[-1,1]^{2}$ or the function $g_{\eps}:A\to\er$ defined as 
\begin{equation*}
g(x) = \left\{
  \begin{array}{l l}
    f_{1}(x), &\text{if}\quad x\in A_{1}\cup B^{\eps}_{1},\\
    f_{2}(x), &\text{if}\quad x\in A_{2}\cup B^{\eps}_{2}.\\
  \end{array} \right.
\end{equation*}
is locally convex on $A.$
Moreover, if $\eps$ is small enough, $g_{\eps}$ is $(L+\alpha)$-Lipschitz and $|g_{\eps}-f|<\alpha$ on $A$.
\end{lem}

\begin{proof}
Direct computation.
\end{proof}

\begin{lem}\label{nakraji}
Let $L,\alpha>0$ and $1>\gamma>\eps>0.$
Let $f$ be a $L$-Lipschitz function on $[-4,4]^{2}\cup [4,5]\times[1,2]$ which is affine on both $[-4,0]\times[-4,4]$ and $[0,4]\times[-4,4]\cup [4,5]\times[1,2]$ 
(and equal to affine functions $f_{1}$ and $f_{2}$, respectively).
Put 
$$
A_{1}=[0,\gamma]\times[-3,-2], A_{2}=[\gamma,\gamma+\eps]\times[-3,0], A_{3}=[\gamma-\eps,\gamma]\times[-1,2], 
$$
$$
A_{4}=[\gamma,4]\times[1,2], B_{1}=[-4,0]\times[-4,4], B_{2}=[4,5]\times[1,2],
$$ 
and
$$
A=A_{1}\cup A_{2}\cup A_{3}\cup A_{4}\cup B_{1}\cup B_{2}.
$$
Then either $f$ is locally convex on $[-4,4]^{2}\cup [4,5]\times[1,2]$ or the function 
\begin{equation*}
g(x) = \left\{
  \begin{array}{l l}
    f_{1}(x), &\text{if}\quad x\in A_{1}\cup A_{2}\cup B_{1},\\
    f_{2}(x)+\frac{f_{1}(\gamma,0)-f_{1}(0,0)-f_{2}(\gamma,0)+f_{1}(0,0)}{\gamma-4}(x\cdot(1,0)-4), &\text{if}\quad x\in A_{3}\cup A_{4},\\
    f_{2}(x), &\text{if}\quad x\in B_{2},\\
  \end{array} \right.
\end{equation*}
is $(L+\alpha)$-Lipschitz, locally convex on $A$ and $|f-g|<\alpha$ on $A,$ if $\eps$ and $\gamma$ are small enough.
\end{lem}

\begin{proof}
Without any loss of generality we can suppose $L=1.$
First we prove that $g$ is continuous on $A.$
To do this we need to prove that
\begin{equation}\label{cont1}
f_{1}(\gamma,a)=f_{2}(\gamma,a)+\frac{f_{1}(\gamma,0)-f_{1}(0,0)-f_{2}(\gamma,0)+f_{1}(0,0)}{\gamma-4}((\gamma,a)\cdot(1,0)-4)
\end{equation}
whenever $(\gamma,a)\in A$ and that
\begin{equation}\label{cont2}
f_{2}(4,a)=f_{2}(4,a)+\frac{f_{1}(\gamma,0)-f_{1}(0,0)-f_{2}(\gamma,0)+f_{1}(0,0)}{\gamma-4}((4,a)\cdot(1,0)-4)
\end{equation}
whenever $(4,a)\in A.$
Define an affine function $f_{3}$ on $\er^{2}$ as
$$
f_{3}(u,v)=\frac{f_{1}(\gamma,0)-f_{1}(0,0)-f_{2}(\gamma,0)+f_{1}(0,0)}{\gamma-4}((u,v)\cdot(1,0)-4).
$$
To prove (\ref{cont1}) we can write
\begin{equation*}
\begin{aligned}
g(\gamma,a)=&f_{2}(\gamma,a)+f_{3}(\gamma,a)\\
=&f_{2}(\gamma,a)+\frac{f_{1}(\gamma,0)-f_{1}(0,0)-f_{2}(\gamma,0)+f_{1}(0,0)}{\gamma-4}\cdot(\gamma-4)\\
=&f_{2}(\gamma,a)+f_{1}(\gamma,0)-f_{1}(0,0)-f_{2}(\gamma,0)+f_{2}(0,0)\\
=&f_{2}(\gamma,a)+f_{1}(\gamma,a)-f_{1}(0,a)-f_{2}(\gamma,a)+f_{2}(0,a)\\
=&f_{2}(\gamma,a)+f_{1}(\gamma,a)-f_{1}(0,a)-f_{2}(\gamma,a)+f_{1}(0,a)=f_{1}(\gamma,a).
\end{aligned}
\end{equation*}
To prove (\ref{cont2}) we can write
\begin{equation*}
\begin{aligned}
g(4,a)=&f_{2}(4,a)+f_{3}(4,a)\\
=&f_{2}(4,a)+\frac{f_{1}(\gamma,0)-f_{1}(0,0)-f_{2}(\gamma,0)+f_{1}(0,0)}{\gamma-4}(4-4)
=f_{2}(4,a).
\end{aligned}
\end{equation*}
Next note that since both $f_{1}$ and $f_{2}$ are $1$-Lipschitz we have
\begin{equation}\label{lip0}
g \text{ is $1$-Lipschitz on } B_{1}\cup A_{1}\cup A_{2},
\end{equation}
and
\begin{equation}\label{lip1}
g \text{ is $1$-Lipschitz on } B_{2},
\end{equation}
also since additionally $f_{3}$ is constant on all lines parallel to $y$-axis and since
$$
\frac{f_{3}(\gamma,0)-f_{3}(4,0)}{4-\gamma}\leq\frac{f_{1}(\gamma,0)-f_{1}(0,0)-f_{2}(\gamma,0)+f_{2}(0,0)-0}{3}\leq \frac{2\gamma }{3}\leq\gamma.
$$
we have
\begin{equation}\label{lip2}
g \text{ is $(1+\gamma)$-Lipschitz on }  A_{4}\cup A_{3}.
\end{equation}
and
\begin{equation}\label{lip3}
|g-f_{2}|\leq 4\gamma \text{ on }  A_{4}\cup A_{3}.
\end{equation}
Now, if $x\in B_{1}$ and $y\in A_{3}$ then $g(x)=f_{1}(x)$, $|g(y)-f_{1}(y)|\leq 3\eps$ and $|x-y|\geq \gamma-\eps$ and therefore
$$
|g(x)-g(y)|\leq |g(x)-f_{1}(y)|+|f_{1}(y)-g(y)|\leq |x-y|+ 3\eps\leq \frac{\gamma+2\eps}{\gamma-\eps}.
$$
So
\begin{equation}\label{lip4}
g \text{ is $\frac{\gamma+2\eps}{\gamma-\eps}$-Lipschitz on }  B_{1}\cup A_{3}.
\end{equation}
If $x\in B_{1}$ and $y\in A_{4}$ then $g(x)=f_{1}(x)$, $f(y)\leq g(y)\leq f_{1}(y)$ and therefore
\begin{equation}\label{lip5}
g \text{ is $1$-Lipschitz on }  B_{1}\cup A_{4}.
\end{equation}
Using (\ref{lip1}) and (\ref{lip2}) and continuity of $g$ we obtain that
\begin{equation}\label{lip6}
g \text{ is $(1+\gamma)$-Lipschitz on }  A_{2}\cup A_{3} \text{ and on } B_{2}\cup A_{4}.
\end{equation}
Finally, if $x\in A_{1}\cup A_{2}$ and $y\in A_{4}\cup B_{2}$ or $x\in A_{1}$ and $y\in A_{3}\cup A_{4}\cup B_{2}$ we have
\begin{equation}\label{lip7}
|g(x)-f_{2}(x)|\leq 2(\gamma+\eps)\leq 4\gamma,\;\; 
|g(y)-f_{2}(y)|\leq 4\gamma
\end{equation}
and $|x-y|\geq1.$
This implies
\begin{equation}
\begin{aligned}
|g(x)-g(y)|\leq &|g(x)-f_{2}(x)|+|f_{2}(x)-f_{2}(y)|+|f_{2}(y)-g(y)|\\
\leq &4\gamma+|x-y|+4\gamma\leq (1+8\gamma)|x-y|.
\end{aligned}
\end{equation}

Now, according to (\ref{lip0}), (\ref{lip1}), (\ref{lip2}), (\ref{lip3}), (\ref{lip4}), (\ref{lip5}), (\ref{lip6}) and (\ref{lip7}) it is sufficient to choose $\frac{\alpha}{4}>\gamma>\eps>0$ small enough such that
$$
\max\left(1+8\gamma,\frac{\gamma+2\eps}{\gamma-\eps}\right)<1+\alpha
$$
to obtain that $g$ is $(1+\alpha)$-Lipschitz on $A$ and $|f-g|<\alpha$ on $A.$
\end{proof}

\begin{lem}\label{had}
Under the assumptions of Lemma~\ref{points} there is a $\frac{1}{2}>\kappa>0$, $R\subset P^{\circ}\cap \er\times(-\kappa,\kappa)$ and a function $h:(P^\eps\setminus P)\cup R\to\er$ such that:
\begin{itemize}
\item[(a)] $R\in\Q$,
\item[(b)] $h=f$ on $P^\eps\setminus P^{\circ}$,
\item[(c)] $h$ is locally convex on $(P^\eps\setminus P^{\circ})\cup R$,
\item[(d)] $(P^\eps\setminus P)\cup R$ is connected,
\item[(e)] $h$ is piecewise affine on $(P^\eps\setminus P^{\circ})\cup R$,
\item[(f)] $h$ is $(L+\delta)$-Lipschitz.
\end{itemize} 
\end{lem}

\begin{proof}
Without any loss of generality we can suppose $L=1.$
Let $\kappa$, $z_i$ $g_\kappa$ as in Lemma~\ref{points}, but with $\frac{\delta}{2}$ in the place of $\delta.$
Consider the sets
$$
X=[-4,4]^{2}\cup [4,5]\times[1,2]\quad\text{and}\quad Y=[-1,1]^{2}.
$$

Find similarities $\Psi_i:\er^2\to\er^2$, $i=0,...,n$ such that if we put $M_i=\Psi_i(X),$ $i=0,n$ and $M_i=\Psi_i(Y),$ $i=1,...,n-1$ we have
\begin{itemize}
\item[(A)] $M_i\cap M_j=\emptyset$ if $i\not=j,$
\item[(B)] $\Psi_0([-4,0]\times[-4,4])\subset\P^{\eps}\setminus P^{\circ},$
\item[(C)] $\Psi_{n}([-4,0]\times[-4,4])\subset\P^{\eps}\setminus P^{\circ},$
\item[(D)] $M_i\subset \er\times(-\kappa,\kappa),$
\item[(E)] $[z^-_i,z^+_i]\subset\Psi_i(\{0\}\times\er),$
\item[(F)] $\Psi_i$ preserves orientation for $i=1,...,n-1$
\end{itemize}

Put $\Omega=\min_{i\not=j}\dist(M_i,M_j),$ note that $\Omega>0$ due to property $(A).$
Define 
$$
T_i:=\co\{\Psi_i((1,\frac{1}{2}),\Psi_i)(1,1),\Psi_{i+1}((-1,-\frac{1}{2}),\Psi_{i+1})(-1,-1)\},
$$
for $i=1,...,n-2$,
$$
T_0:=\co\{\Psi_0(5,1),\Psi_0(5,2),\Psi_{1}(-1,-\frac{1}{2}),\Psi_{1}(-1,-1)\}
$$
and 
$$
T_{n-1}:=\co\{\Psi_n(5,1),\Psi_n(5,2),\Psi_{n-1}(1,\frac{1}{2}),\Psi_{n-1}(1,1)\}.
$$
and put 
\begin{equation}\label{sjednoceni}
R:=\left(\bigcup_{i=0}^{n-1} T_i\right) \cup \left(\bigcup_{i=0}^{n} M_i\right).
\end{equation}

Let $\rho_i$ be scaling ratio of $\Psi_i.$
Let $g_i,$ $i=1,...,n-1$ be the function $g$ from Lemma~\ref{uvnitr} with $\alpha=\frac{\Omega\delta\rho_i}{4}$ (and corresponding $\eps$) and with $f_1(x)=\rho_i\kappa\circ\Psi_i$ and $f_2(x)=\rho_i\kappa\circ\Psi_{i}$ (with the exception if $g_{\kappa}$ is already convex on $M_i$, in which case we put $g_i=g_\kappa|_{M_{i}}$),
let $g_0$ be the function $g$ from Lemma~\ref{nakraji} with $\gamma=\frac{\Omega\delta\rho_i}{4}$ (and corresponding $\eps$ and $\gamma$) and with $f_1=\rho_0\kappa\circ\Psi_0$ and 
$f_2=\rho_0\kappa\circ\Psi_{0}$
and finally, let $g_n$ be the function $g$ from Lemma~\ref{nakraji} with $\gamma=\frac{\Omega\delta\rho_i}{4}$ (and corresponding $\eps$ and $\gamma$) and with $f_1=\rho_n\kappa\circ\Psi_n$ and 
$f_2=\rho_n\kappa\circ\Psi_{n}$.

Consider now the function $h$ defined by the formula
$$
h=
\begin{cases}
\frac{1}{\rho_i}g_i\circ\Psi_i^{-1} \quad\text{on}\quad M_i\\
g_\kappa \quad\text{otherwise}.
\end{cases}
$$ 

Property $(a)$ follows from (\ref{sjednoceni}) and the fact that every $M_i$ and every $T_i$ is a polygon. 
Properties $(b),(c)$ and $(e)$ follow directly from the construction and corresponding properties of the functions $g_i$ and property $(d)$ is obvious.
We will now finish the proof by proving property $(f).$

So suppose that $a,b\in (P^\eps\setminus P)\cup R.$ we need to prove that $|h(a)-h(b)|\leq (1+\delta)|a-b|.$
We can additionally suppose that either $a$ or $b$ belongs to some $M_i$ since otherwise there is nothing to prove.
We will prove only the case $a\in M_i$, $b\in M_j$, $i\not=j$, the other cases can be proved following the same lines.
By Lemma~\ref{uvnitr} (for $i=1,...,n-1$) and Lemma~\ref{nakraji} (for $i=0,n$) we can now write
$$
\begin{aligned}
|h(a)-h(b)|\leq &|h(a)-g_\kappa(a)|+|g_\kappa(a)-g_\kappa(b)|+|g_\kappa(b)-h(b)|\\
< &\frac{1}{\rho_i}\cdot\frac{\Omega\delta\rho_i}{4}+\left(1+\frac{\delta}{2}\right)\cdot|a-b|+\frac{1}{\rho_j}\cdot\frac{\Omega\delta\rho_j}{4}\\
\leq &\frac{\delta}{2}|a-b|+\left(1+\frac{\delta}{2}\right)\cdot|a-b|= (1+\delta)|a-b|,
\end{aligned}
$$
which is what we need.
\end{proof}

\begin{proof}[Proof of Lemma~\ref{keylemma}]
Without any loss of generality we can suppose $L=1.$
Let $V$ be the set of all points $v\in\partial P$ with the property that there is some $\eps_v>0$ such that 
$P\cap B(v,\eps_v)$ is similar to $\{(x,y):x\geq 0\}\cap B(0,1)$ and that $f$ is affine on $P\cap B(v,\eps_v)$.
Since $P\in\Q$, the set $\partial P\setminus V$ is finite and we can without any loss of generality assume that $l(\eps)\cap (\partial P\setminus V)=\emptyset.$

This means that the closure of every bounded component $C_i$ of $P\cap l(\eps)$ is a similar copy of 
$$
\co\{(-1,a_i),(-1,b_i),(1,c_i),(1,d_i)\}=:P_i
$$ 
for some $a_i<b_i$, $c_i<d_i$ and such that for some $\eps_i>0$
$f$ is locally affine on $P_i^{\eps_i}\setminus P$,
where 
$$P_i^{\eps_i}:=\co\{(-1,a_i-\eps_i),(-1,b_i+\eps_i),(1,c_i-\eps_i),(1,d_i+\eps_i)\}.
$$
Then 
$$
\alpha=\min_{i\not= j}\dist(C_i,C_j)>0
$$
Let $\Psi_i$ be a similarity between $C_i$ and $S_i$ and let $\kappa_i$, $R_i$ and $h_i$ be $\kappa$, $R$ and $h$ as obtained from Lemma~\ref{had} for
$\eps=\eps_i$, $P=P_i$, $f=\rho_i g\circ \Psi_i$ and $\delta=\frac{\min(\alpha,\eps_i,1)\rho_i\eps}{4}$, where $\rho_i$ is the similarity ratio on $\Psi_i.$

Put $Q=P\setminus(\bigcup R_i)$ and define $\tilde h:Q^c\to\er$ by
$$
\tilde h=
\begin{cases}
\frac{1}{\rho_i}h_i\circ\Psi_i^{-1} \quad\text{on}\quad R_i\\
g \quad\text{otherwise}.
\end{cases}
$$

Let $K$ be the Lipschitz constant of $\tilde h,$ the using the Kirszbraun theorem on extensions of Lipschitz functions we can find a $K$-Lipschitz function $h$ on $\er^2$ such that
$h=\tilde h$ on $P^c.$

Now, property $(1)$ follows directly form the definition of $Q$ and $(a)$ in Lemma~\ref{had}, 
property $(2)$ from the definition of $h$ and $(b)$ in Lemma~\ref{had} 
and property $(3)$ from $(d)$ in Lemma~\ref{had}.

It remains to prove that the pair $(Q,h)$ is $(1+\eps)$-good.
The local convexity and piecewise affinity of $h$ on $Q^c$ follows from $(c)$ and $(e)$ in Lemma~\ref{had} and the corresponding properties of $g$,
so the proof will be finished, if we verify that $K\leq(1+\eps).$

To do this pick $a,b\in\er^2$, we need to prove that $|h(a)-h(b)|\leq (1+\eps)|a-b|.$

We can additionally suppose that either $a$ or $b$ belongs to some $R_i$ since otherwise there is nothing to prove.
We will prove only the case $a\in R_i$, $b\in R_j$, $i\not=j$, the other cases can be proved following the same lines.

Using the definition of $h$, namely property $(f)$ from Lemma~\ref{had} we can now write
$$
\begin{aligned}
|h(a)-h(b)|= &|h_i(a)-h_j(b)|\leq |h_i(a)-f(a)|+|f(a)-f(b)|+|f(b)-h_j(b)|\\
\leq &\frac{1}{\rho_i}\cdot\frac{\min(\alpha,\eps_i)\rho_i\eps}{4}+\left(1+\frac{\eps}{4}\right)\cdot|a-b|+\frac{1}{\rho_j}\cdot\frac{\min(\alpha,\eps_j)\rho_j\eps}{4}\\
\leq &\frac{2\eps}{4}|a-b|+\left(1+\frac{\delta}{2}\right)\cdot|a-b|< (1+\delta)|a-b|.
\end{aligned}
$$

\end{proof}

 {\bf Acknowledgment}. I would like to thank Professor Lud\v ek Zaj\' i\v cek for finding all the historical information and to Professor Ji\v r\' i Jel\' inek for translating the original argument by Pasqualini and also for many comments on the previous versions of the manuscript.


\begin{thebibliography}{99}

\bibitem{BuZa} Burago, Ju. D.; Zalgaller,V. A.: {\it Sufficient tests for convexity.} Zap. Naucn. Sem. Leningrad. Otdel. Mat. Inst. Steklov., {\bf 45} (1974), 3--52
\bibitem{Dm} Dmitriev, V. G.: {\it On the construction of $\H_{n-1}$-almost everywhere convex hypersurface in $\er^{n+1}$.} Mat. Sb. (N.S.), {\bf 114(156)} (1981), 511–-522
\bibitem{Pa} Pasqualini, L.: {\it Sur les conditions de convexit\'e d'une vari\'et\'e.} Ann. Fac. Sci. Toulouse Sci. Math. Sci. Phys. (4), no. 2, (1938), 1--45




\end{thebibliography}
\end{document}